\documentclass[12pt,oneside]{amsart}

\usepackage{amssymb,latexsym}

\usepackage{hyperref}
\hypersetup{
  colorlinks   = true, %Colours links instead of ugly boxes
  urlcolor     = blue, %Colour for external hyperlinks
  linkcolor    = blue, %Colour of internal links
  citecolor   = red %Colour of citations
}
\usepackage{anysize}

\usepackage{pgfplots}
\pgfplotsset{compat=1.15}
\usepackage{mathrsfs}
\usetikzlibrary{arrows}
\usetikzlibrary{calc}
\usetikzlibrary{intersections}
\usetikzlibrary{through}

\usepackage{mathtools}

\newtheorem{theorem}{Theorem}[section]

\newtheorem{lemma}[theorem]{Lemma}
\newtheorem{claim}[theorem]{Claim}
\newtheorem{conjecture}[theorem]{Conjecture}
\newtheorem{corollary}[theorem]{Corollary}

\theoremstyle{definition}
\newtheorem{definition}[theorem]{Definition}

\theoremstyle{remark}
\newtheorem{remark}[theorem]{Remark}

\numberwithin{equation}{section}

\newcommand{\ZQ}{Z(Q_k)}
\newcommand{\Zg}{Z(g_{n,k})}
\newcommand{\MQ}{M(Q_k)}
\newcommand{\spcap}[1]{C_{#1}}

\renewcommand{\epsilon}{\varepsilon}
\renewcommand{\phi}{\varphi}
\renewcommand{\kappa}{\varkappa}

\begin{document}

\title{Covering by planks and avoiding zeros of polynomials}

\author{Alexey Glazyrin{$^\spadesuit$}}
\author{Roman Karasev{$^\clubsuit$}}
\author{Alexandr Polyanskii{$^\diamondsuit$}}

\thanks{{$^\spadesuit$} Partially supported by the NSF grant DMS-2054536}
\thanks{{$^\diamondsuit$} Supported by the Young Russian Mathematics award and the program ``Leading Scientific Schools'' through Grant No. NSh-775.2022.1.1.} 

\address{Alexey Glazyrin, School of Mathematical \& Statistical Sciences, The University of Texas Rio Grande Valley, Brownsville,
TX 78520, USA}
\email{alexey.glazyrin@utrgv.edu}

\address{Roman Karasev, Institute for Information Transmission Problems RAS, Bolshoy Karetny per. 19, Moscow, Russia 127994}
\email{r\_n\_karasev@mail.ru}
\urladdr{http://www.rkarasev.ru/en/}

\address{Alexandr Polyanskii, Moscow Institute of Physics and Technology, Institutskiy per. 9, Dolgoprudny, Russia 141700}
\email{alexander.polyanskii@yandex.ru}


\begin{abstract}
We note that the recent polynomial proofs of the spherical and complex plank covering problems by Zhao and Ortega-Moreno give some general information on zeros of real and complex polynomials restricted to the unit sphere. As a corollary of these results, we establish several generalizations of the celebrated Bang plank covering theorem. 

We prove a tight polynomial analog of the Bang theorem for the Euclidean ball and an even stronger polynomial version for the complex projective space. Specifically, for the ball we show that for every real nonzero $d$-variate polynomial $P$ of degree $n$, there exists a point in the unit $d$-dimensional ball at distance at least $1/n$ from the zero set of the polynomial $P$.

Using the polynomial approach, we also prove the strengthening of the Fejes T\'oth zone conjecture on covering a sphere by spherical segments, closed parts of the sphere between two parallel hyperplanes. In particular, we show that the sum of angular widths of spherical segments covering the whole sphere is at least $\pi$.
\end{abstract}

\maketitle

\section{Introduction}

Bang's plank covering theorem \cite{bang1951} in its particular case\footnote{In this paper we only consider Bang's theorem for covering a Euclidean ball and do not discuss its general version for covering a convex body of given width.} asserts that \emph{if a Euclidean ball of arbitrary dimension is covered by planks then the sum of the widths of the planks is greater or equal to the diameter of the ball.} Here a \emph{plank of width $2\delta$} is a Euclidean metric $\delta$-neighborhood of an affine hyperplane in $\mathbb R^d$. When the planks have equal widths, the Bang theorem may be restated as follows: \emph{For any collection of $n$ hyperplanes in the Euclidean space, there is a point in the unit ball at distance at least $1/n$ from the union of the hyperplanes.}

There are spherical \cite{polyajiang2017} and complex \cite{ball2001} versions of the plank covering theorem, with similar statements in the case of equal planks. The equal plank cases of those versions were recently proved in \cite{zhao2021,ortegamoreno2021} by maximizing the absolute value of the polynomial, whose zero set is the union of hyperplanes (the paper \cite{ortegamoreno2021} also addresses the case of unequal complex planks that we extend in Theorem \ref{theorem:ortega-moreno2} below). 

We start with two statements whose proofs essentially follow from \cite{zhao2021,ortegamoreno2021}.

\begin{theorem}
\label{theorem:zhao-inequality}
If a polynomial $P\in \mathbb R[x_1, \ldots, x_d]$ of degree $n$ has a nonzero restriction to the unit sphere $S^{d-1}\subset\mathbb R^d$ and attains its maximal absolute value on $S^{d-1}$ at a point $p$ then $p$ is at angular distance at least $\frac{\pi}{2n}$ from the intersection of the zero set of $P$ with~$S^{d-1}$.
\end{theorem}

\begin{theorem}
\label{theorem:ortega-moreno}
If a homogeneous polynomial $P\in \mathbb C[z_1, \ldots, z_d]$ of degree $n$ is not identically zero and attains its maximal absolute value on the unit sphere $S^{2d-1}\subset\mathbb C^d$ at a point $p$ then $p$ is at angular distance at least $\arcsin\frac{1}{\sqrt{n}}$ from the intersection of the zero set of~$P$ with $S^{2d-1}$. 
\end{theorem}

We also need the following description of the equality case from Theorem~\ref{theorem:zhao-inequality}.

\begin{theorem}	
\label{theorem:zhao-equality}
If a polynomial $P\in \mathbb R[x_1, \ldots, x_d]$ of degree $n$ has nonzero restriction to the unit sphere $S^{d-1}\subset\mathbb R^d$ and attains its maximal absolute value on $S^{d-1}$ at a point $p$ and $p$ is precisely at angular distance $\frac{\pi}{2n}$ from the intersection of the zero set of $P$ with $S^{d-1}$ then there exists a circle $\Sigma\subseteq S^{d-1}$ through $p$ centered at the origin such that the $2n$ zeros and the $2n$ maxima of the absolute value of the restriction $P|_\Sigma$ interlace and split the circle $\Sigma$ in $4n$ segments of length $\frac{\pi}{2n}$ each.
\end{theorem}

As for the original version of the problem about the unit Euclidean ball, its maximization variant is harder to formulate. The natural question is: \emph{Let $P\in \mathbb R[x_1, \ldots, x_d]$ be a polynomial of degree $n$. How to find a point $p\in B^d\subset\mathbb R^d$ in the unit ball that is at distance at least $\frac{1}{n}$ from the intersection of the zero set of the polynomial $P$ with the ball $B^d$?}

Note that the maximum point of $|P|$ may not serve as $p$. For example, the Chebyshev polynomial $T_n(x) = \cos (n\arccos x)$ attains its maximal absolute value at the ends of the interval $[-1,1]$, but the distance from this maximum to the closest zero is $1 - \cos\frac{\pi}{2n}$, which is of order $1/n^2$. This is a counterexample to the naive approach for $d=1$. For higher dimensions, one can apply $T_n$ to one of the Euclidean coordinates on $B^d$. 

Hence, the choice of a point avoiding zeros must be trickier. A relatively simple argument allows us to give a partial answer to the question with the right order of magnitude.

\begin{theorem}
\label{theorem:ball-weak2}
There exists an absolute constant $C>0$ with the following property. For every nonzero $P\in \mathbb R[x_1, \ldots, x_d]$ of degree $n$, let the pair of points 
\[
(p,q)\in S^{2d-1}\subset \mathbb R^d\times \mathbb R^d
\]
maximize the absolute value of the product $P(x)P(y)$ over $(x,y)\in S^{2d-1}$. Then at least one of $p\in B^d$ and $q\in B^d$ is at distance at least $\frac{C}{n}$ from the zero set of the polynomial $P$. 
\end{theorem}

A more careful reduction to Theorems \ref{theorem:zhao-inequality} and \ref{theorem:zhao-equality} in the proof of the following theorem allows us to find an answer to the question with the tight constant $1/n$.

\begin{theorem}
\label{theorem:ball-strong}
There exists a sequence of even univariate analytic functions $G_n : \mathbb R\to\mathbb R$ with the following property. For every nonzero polynomial $P\in \mathbb R[x_1, \ldots, x_d]$ of degree $n$, some of the points of $B^d\subset\mathbb R^d$, where the absolute value of the expression $P(x)G_n(|x|)$ attains its maximum on the ball $B^d$, is at distance at least $\frac{1}{n}$ from the zero set of the polynomial $P$.
\end{theorem}

%\begin{remark}
%It is not clear if one can change the quantifier, ``some of the points'' to ``any of the points'', in this theorem.
%\end{remark}

\begin{remark}
The precise definition of the multiplier $G_n$ is given in \eqref{equation:g_n}, \eqref{equation:bigg} of the proofs section. Below we explicitly state the result that directly follows from this theorem without mentioning the multiplier.
\end{remark}

\begin{corollary}
\label{corollary:ball-strong}
For every nonzero polynomial $P\in \mathbb R[x_1, \ldots, x_d]$ of degree $n$, there exists a point of $B^d\subset\mathbb R^d$ at distance at least $\frac{1}{n}$ from the zero set of the polynomial $P$.
\end{corollary}

Thinking about a statement generalizing both Bang's theorem and the results on zeros of polynomials, we formulate the following conjecture.

\begin{conjecture}
\label{conjecture:polynomial-planks}
Assume that $P_1,\ldots, P_N\in \mathbb R[x_1, \ldots, x_d]$ are nonzero polynomials and $\delta_1,\ldots, \delta_N > 0$ are such that
\[
\sum_{k=1}^N \delta_k \deg P_k \le 1.
\]
Then there exists a point $p\in B^d\subset\mathbb R^d$ such that, for every $k=1,\ldots, N$, the point $p$ is at distance at least $\delta_k$ from the zero set of $P_k$.
\end{conjecture}

\begin{remark}
Equality cases in Theorem \ref{theorem:ball-strong} and Conjecture \ref{conjecture:polynomial-planks} are not just those defined by hyperplane configurations from the Bang theorem. Apart from sets of parallel hyperplanes, the bounds are sharp for unions of cylinders over concentric spheres of any dimension or even combinations of cylinders over spheres of different dimensions. In particular, for $n=2m$, the bound of Theorem \ref{theorem:ball-strong} is sharp for $P=\prod\limits_{i=1}^m (x_1^2+\ldots+x_{k_i}^2-(\frac {2i-1} {2m})^2)$, where $\{k_i\}$ is an arbitrary increasing sequence of positive integers no greater than $d$. We do not know whether all equality cases can be obtained in a similar manner.
\end{remark}

In order to justify Conjecture \ref{conjecture:polynomial-planks} we note that \cite{ortegamoreno2021} essentially provides an approach to the complex projective version of this conjecture that we state and prove explicitly.

\begin{theorem}
\label{theorem:ortega-moreno2}
Assume that $P_1,\ldots, P_N\in \mathbb C[z_1, \ldots, z_d]$ are nonzero homogeneous polynomials and $\delta_1,\ldots, \delta_N > 0$ are such that
\[
\sum_{k=1}^N \delta_k^2 \deg P_k \le 1.
\]
Then the point of maximum of the absolute value of $P_1^{\delta_1^2}\cdots P_N^{\delta_N^2}$ on the unit sphere $S^{2d-1}\subset\mathbb C^d$ is, for every $k$, at angular distance at least $\arcsin\delta_k$ from the intersection of the zero set of $P_k$ with $S^{2d-1}$.
\end{theorem}

\begin{remark}
Since polynomials in the above theorem are homogeneous, one may change the conclusion to more classical \cite{ball2001}: \emph{For every $k$, $p$ is at Euclidean distance at least $\delta_k$ from the zero set of $P_k$.}
\end{remark}

\begin{remark}
In Theorem \ref{theorem:ortega-moreno} the requirement that the complex polynomial is homogeneous is essential. Otherwise the polynomial $P(z_1,\ldots, z_d) = z_1^n + 1$ attains the maximal absolute value on the unit sphere at points with $z_1 = e^{i\frac{2k}{n}\pi}$ (and all other $z_i=0$), $k=0,1,\ldots, n$. On the other hand, $P$ has zeros whenever $z_1 = e^{i\frac{2k+1}{n}\pi}$, at angular distance $\pi/n$ from their neighboring points of maximum. Since for sufficiently large $n$, $\pi/n < \arcsin\frac{1}{\sqrt n}$, the statement is not true in the nonhomogeneous case. Analogously, Theorem \ref{theorem:ortega-moreno2} would fail for the set of $n$ nonhomogeneous linear polynomials $P_k(z_1,\ldots,z_d) = z_1 - e^{i\frac{2k+1}{n}\pi}$, $k=0, 1,\ldots, n-1$, $\delta_k = \frac{1}{\sqrt n}$.
\end{remark}

\begin{remark}[Communicated by Fedor Petrov]
One can also formulate the real spherical version of the conjecture as an extension of Theorem \ref{theorem:zhao-inequality}. The approach of Zhao (see Lemma \ref{lemma:zhao} below) would have a potential to work directly if we could claim that any point of maximum of the absolute value of a trigonometric polynomial of degree $n$ is at distance at least $\frac{\pi k}{2n}$ from a zero of multiplicity $k$. Unfortunately, this is false as the trigonometric polynomial of degree $n=2$
\[
T(x) = (1 - \cos x) (0.9 + \cos x)
\]
has a zero of multiplicity $k=2$ at the origin. Its derivative vanishes when $\cos x = 0.05$ or $\sin x = 0$, and it is easy to see that the distance from the double zero to the point of maximum is $\arccos 0.05 < \frac{\pi}{2} = \frac{\pi k}{2 n}$.
\end{remark}

%Going back to the classical covering by planks or zones, we prove a corollary of Theorem~\ref{theorem:zhao} generalizing the result of \cite{polyajiang2017} and confirming \cite[Conjecture~2]{polyajiang2017}.

\begin{definition}
A \emph{spherical segment of width $2\delta$} on the unit sphere $S^{d-1}$ is a closed $\delta$-neighborhood (in the intrinsic metric of the sphere) of an intersection of $S^{d-1}\subset\mathbb R^d$ with an affine hyperplane $H\subset \mathbb R^d$. 
\end{definition}

In case a hyperplane $H$ contains the origin, a spherical segment is called \textit{a zone}. In 1973, Fejes T\'{o}th conjectured \cite{toth1973} that if $n$ equal zones cover the sphere then their width is at least $\pi/n$. He also formulated the generalized conjecture that the total width of any set of zones covering the sphere is at least $\pi$. The generalized conjecture was proved by Jiang and Polyanskii in \cite{polyajiang2017} based on the approaches of Bang \cite{bang1951} and Goodman-Goodman \cite{goodman1945} (see also \cite{polyanskii2021cap} for a stronger version of the zone conjecture). In \cite{ortega2021optimal}, Ortega-Moreno found another proof of the Fejes T\'{o}th conjecture for equal zones that was recently streamlined by Zhao \cite{zhao2021}. As a corollary of our results, we show how to use this approach to prove the generalized zone conjecture. Moreover, we prove that the generalized conjecture also holds for spherical segments, thus resolving Conjectures 1 and 2 of Jiang and Polyanskii \cite{polyajiang2017}.

%Note that unlike the zones in the main result of \cite{polyajiang2017}, the spherical segments need not be centrally symmetric about the origin. 

\begin{corollary}
\label{corollary:pj2}
If the unit sphere $S^{d-1}$, $d\ge 2$, is covered by a set of spherical segments, then the sum of their widths is at least $\pi$.
\end{corollary}

\begin{remark}
Similar to the proof of this corollary, Theorem \ref{theorem:ball-strong} (in the form of Corollary~\ref{corollary:ball-strong}) implies Bang's theorem on covering of a ball by planks. We leave this implication as an exercise to the reader.
\end{remark}

We recall one more conjecture that is still open. 

\begin{conjecture}
If $d\ge 4$ and the unit sphere $S^{d-1} \subset \mathbb R^d$ is covered by a finite number of planks then the sum of their widths is at least $2$.
\end{conjecture}

\begin{remark}
The case $d=2$ of this conjecture obviously fails. The case $d=3$ is known from \cite{moese1932} by calculating areas of spherical segments. When $d\ge 4$ and the number of planks is at most $3$, the conjecture reduces to the result of \cite{moese1932}, similarly to the proofs of \cite[Theorem~5.3 and Corollary~5.4]{bezdek2009tarski} with integration over spheres instead of balls. The general case $d\ge 4$ of the conjecture is unlikely to follow from the polynomial approach because the sphere itself is defined by the polynomial of degree $2$. Moreover, the proofs below are based on reducing the more general statement to $d=2$ but this conjecture fails for $d=2$.
\end{remark}

\subsection*{Acknowledgments} The authors thank Arseniy Akopyan, Fedor Petrov, Alexey Balitskiy, Vladimir Dol'nikov, Zilin Jiang, Danil Skuridin, and the anonymous referee for discussions and useful remarks.

\section{Proof of the real spherical case}

The proof is based on the general version of the lemma that was essentially proved in \cite{zhao2021}. We present here the full argument for a reader's convenience. Unlike in \cite{zhao2021}, here we do not assume that a polynomial is homogeneous, which eventually allows us to deduce Corollary \ref{corollary:pj2}, as well as Theorems \ref{theorem:ball-weak2} and \ref{theorem:ball-strong}.

\begin{lemma}[Zhao, \cite{zhao2021}]
\label{lemma:zhao}
If a polynomial $P\in \mathbb R[x_1, x_2]$ of degree $n$ has nonzero restriction to the unit circle $S^1\subset\mathbb R^2$ and attains its maximal absolute value on $S^1$ at a point $p$ then $p$ is at angular distance at least $\frac{\pi}{2n}$ from the intersection of the zero set of $P$ with $S^1$. If the distance from $p$ to a zero of $P$ is precisely $\frac{\pi}{2n}$ then $2n$ zeros and $2n$ maxima of the absolute value of $P$ interlace and split the circle $S^1$ into $4n$ intervals of length $\frac{\pi}{2n}$ each.
\end{lemma}

\begin{proof}
We switch to the polar coordinate on the circle and consider $P$ as a trigonometric polynomial of degree $n$. From here on $P(t)$ is a trigonometric polynomial in $t\in\mathbb R$ with period $2\pi$ and degree $n$.

Any such polynomial has at most $2n$ zeros (counted with multiplicities) over its period $2\pi$, or is identically zero. This is true since the zeros are the intersections of the zero set of $P$ with the circle $S^1$ (curve of degree $2$). 
	
After a suitable shift, we assume that $P$ has maximal absolute value $M$ at the origin. Consider the trigonometric polynomial of degree at most $n$
\[
Q(t) = P(t) \pm M \cos nt.
\]
Choosing the sign $\pm$ appropriately we assume that $Q$ has zeros of multiplicity $2$ at the points $2\pi k$, $k\in\mathbb Z$. Apart from this, for $m=1,\ldots, 2n-1$, either the expression
\[
Q\left(\frac{m \pi}{n}\right) = P\left(\frac{m \pi}{n}\right) \pm M \cos m \pi = P\left(\frac{m \pi}{n}\right) \pm M (-1)^m
\]
has sign $\pm(-1)^m$, or $Q$ has a multiplicity $2$ zero at $\pi m/n$. 

Applying the intermediate value theorem carefully, one observes that there are at least $2n-2$ zeros of $Q$ (counted with multiplicity) on $\left[ \frac{\pi}{n}, \frac{(2n-1)\pi}{n} \right]$. Namely, when the values on the endpoints of $\left[ \frac{m \pi}{n}, \frac{(m+1)\pi }{n} \right]$, $m=1,\ldots,2n-2$, are nonzero, we find a zero of $Q$ in the interior. If $Q$ has a zero at a point $\pi m/n$ then at this point both $P$ and $M\cos nt$ are locally extremal and therefore this zero is of multiplicity at least $2$. In this case we assign those (at least two) zeros to the two neighboring intervals.

Since we also have a multiplicity $2$ zero at the origin, we have already identified $2n$ zeros of $Q$ per period. Moreover, if $Q$ has a zero in $\frac{-\pi}{n}$ or $\frac{\pi}{n}$ then the zero is of multiplicity $2$ and $Q$ has to be identically zero.
	
Therefore $Q$ is either identically zero (then $P(t)=\mp M\cos nt$ and the claim holds, including the characterization of the equality case), or $Q$ has no zero on the set
\[
\left[\frac{-\pi}{n}, 0\right)\cup \left(0, \frac{\pi}{n}\right].
\]
In particular, if $M>0$ then on $\left[\frac{-\pi}{n}, \frac{\pi}{n}\right]$ the inequality 
\[
P(t)\ge M \cos nt
\]
holds. If $M<0$ then the opposite inequality holds. In both cases, there are no zeros of $P$ in $\left(\frac{-\pi}{2n}, \frac{\pi}{2n}\right)$ apart from 0 itself.

In the above argument the only possibility to have the distance from a maximum point of $P$ to a zero of $P$ precisely $\frac{\pi}{2n}$ was $Q=0$, $P(t)=\mp M\cos nt$. In this case the $2n$ zeros and the $2n$ maxima of the absolute value of $P$ interleave and split the circle $S^1$ in $4n$ segments of length $\frac{\pi}{2n}$ each.
\end{proof}
	
\begin{proof}[Proof of Theorems \ref{theorem:zhao-inequality} and \ref{theorem:zhao-equality}]
We consider points $p_0$ from the zero set of $P$ and $p_m$ maximizing the absolute value of $P$ such that the angular distance between them is minimal possible. For the linear span of $p_0$ and $p_m$, $\langle p_0, p_m\rangle$, we apply Lemma~\ref{lemma:zhao} and prove that the angular distance between $p_0$ and $p_m$ is at least $\frac{\pi}{2n}$. By the second part of Lemma~\ref{lemma:zhao}, if the angular distance is precisely $\frac{\pi}{2n}$ then $\Sigma=S^{d-1}\cap \langle p_0, p_m\rangle$ is the needed circle.
\end{proof}

\begin{proof}[Proof of Corollary \ref{corollary:pj2}]
Call the \emph{core} of a spherical segment the intersection $H\cap S^{d-1}$, where $H$ is a hyperplane from the definition of a spherical segment. First, assume that spherical segments have equal widths $2\delta$ and there are $n$ spherical segments in total. If
\[
2 \delta n < \pi,
\]
which is equivalent to
\[
\delta< \frac{\pi}{2n},
\]
then consider the polynomial of degree $n$ 
\[
P = L_1\cdot \dots \cdot L_n,
\]
where $L_i=0$ is the linear equation of the $i^{\text{th}}$ core. 

Applying Theorem \ref{theorem:zhao-inequality} to $P$, we obtain a point $x\in S^{d-1}$ at spherical distance at least $\frac{\pi}{2n}> \delta$ from all the cores. This point does not belong to any of the spherical segments, and this case is done. Note that the equal width case works for spherical segments in general, not just zones.

The next case is when all widths of spherical segments are rational with a common denominator $N$. Then we split every spherical segment into several spherical segments of width $1/N$. %It is important here that we may use not necessarily centrally symmetric spherical segments, then the splitting is evidently possible and the width is additive under the splitting.
This reduces the rational case to the case of equal spherical segment widths.

The general case is done by assuming the contrary, that is, the sum of widths of spherical segments is less than $\pi$. We slightly increase the width of each spherical segment so that it becomes rational but the sum of widths is still less than $\pi$. Then applying the rational case, we reach a contradiction.
\end{proof}

\section{Proofs of the complex spherical cases}

The proof is based on the lemma that was essentially proved in \cite{ortegamoreno2021}, but we present here the full argument for a reader's convenience.

\begin{lemma}[Ortega-Moreno, \cite{ortegamoreno2021}]
\label{lemma:ortega-moreno}
If a homogeneous polynomial $P\in \mathbb C[z_1, z_2]$ of degree $n$ is nonzero and attains its maximal absolute value on the unit sphere $S^3\subset\mathbb C^2$ at a point $p$ then $p$ is at angular distance at least $\arcsin\frac{1}{\sqrt{n}}$ from the intersection of the zero set of $P$ with $S^3$. 
\end{lemma}

\begin{proof}
In case $n=1$ $P$ is linear and it is easy to see that the maximum of $P$ is at distance $\arcsin 1 = \pi/2$ from its zero hyperplane. This allows us to assume $n>1$ and divide by $n-1>0$.
	
Note that $|P|$ is a well-defined function on the complex projective line $\mathbb CP^1 = S^3/S^1$, where $S^1\subset \mathbb C$ are the complex numbers of unit norm. From here on we consider $|P|$ as a function of the affine coordinate $w\in\mathbb C$.

Choose an affine coordinate $w=x+iy$ on the projective line so that $|P(0)|=0$. In such coordinates a point $w\in\mathbb C$ corresponds to
\[
\left(\frac{w}{\sqrt{1+|w|^2}}, \frac{1}{\sqrt{1+|w|^2}}\right)\in S^3
\]
and the absolute value we study expresses as
\[
|P| = |w Q(w)| \left(1 + |w|^2\right)^{-n/2},
\]
where $Q$ is a univariate polynomial. After a suitable rotation, we can assume that the maximum of $|P|$ is attained at $a\in\mathbb R^+$. Write down the differential of $\ln |P|$
\[
d\ln |P| = \Re \frac{dw}{w} + \Re\frac{Q'dw}{Q} - n \frac{xdx + ydy}{1 + |w|^2}
\]
and observe that it is 0 at $a$, that is,
\[
\frac{dx}{a} + \Re\frac{Q'dw}{Q} - n \frac{adx}{1 + a^2} = 0\text{, and thus, } d\ln |Q| = \Re\frac{Q'dw}{Q} = \left(\frac{na}{1+a^2} - \frac{1}{a}\right)dx.
\]
If
\[
\frac{na}{1+a^2} - \frac{1}{a} < 0,
\]
which is equivalent to
\[
a^2 < \frac{1}{n-1},
\]
then $\ln |Q|$ grows when $a$ is replaced by $a-\varepsilon$ for a sufficiently small $\varepsilon > 0$. From the maximum principle, $|Q|$ has a value that is strictly greater than $|Q(a)|$ and is attained at some point of the circle $\{|w|=a\}$. But in this case $|P|$ at this point will be strictly greater than $|P(a)|$, contradicting the initial choice of $a$.
	
Hence $a^2\ge \frac{1}{n-1}$. In terms of the angular distances on $S^3$, we need to bound from below the distance from $p_0=(0,1)$ to
\[
p_a=\left(\frac{a}{\sqrt{1+a^2}}, \frac{1}{\sqrt{1+a^2}}\right),
\] 
which equals
\[
\arccos |p_0\cdot p_a| = \arccos \frac{1}{\sqrt{1+a^2}} \ge \arccos \sqrt{\frac{n-1}{n}} = \arcsin \frac{1}{\sqrt{n}}.
\]
\end{proof}

\begin{proof}[Proof of Theorem \ref{theorem:ortega-moreno}]
Assuming the contrary, we find a point $p_0$ such that $P(p_0)=0$ and another point $p_m$, where the maximum of $|P|$ is attained, so that the angular distance between them is strictly smaller than $\arcsin\frac{1}{\sqrt{n}}$. Going to the linear span of $p_0$ and $p_m$, we apply Lemma~\ref{lemma:ortega-moreno} and reach a contradiction.
\end{proof}

\begin{proof}[Proof of Theorem \ref{theorem:ortega-moreno2}]
The proof is similar to the one above and just extends the argument from \cite{ortegamoreno2021}. Consider the product 
\[
P(z) = P_1^{\delta_1^2}\cdots P_N^{\delta_N^2}
\]
and let
\[
D = \sum_{k=1}^N \delta_k^2 \deg P_k \le 1.
\]
Note that $P$ is not well-defined on the unit sphere $S^{2d-1}$ (as the powers of complex numbers are multi-valued), but its absolute value $|P|$ is well-defined and descends to a function on $\mathbb CP^{d-1}$. Let $p$ be a maximum point of $|P|$ on the sphere. Consider a zero $p_0$ of $P_k$ and show that its angular distance from $p$ to $p_0$ is at least $\arcsin\delta_k$.

After restricting to the two-dimensional subspace through $p$ and $p_0$, we may switch from $S^3\subset\mathbb C^2$ to its quotient $\mathbb CP^1$ and an affine coordinate $w=x+iy$ there. Assuming $p_0=0$, we get
\[
|P| = \left|w^{\delta_k^2} Q(w)\right| \left(1 + |w|^2\right)^{-D/2},
\]
where $Q(w)$ is a product of expressions of the form $R_\ell(w)^{\delta_\ell^2}$ with polynomial $R_\ell$. Note that $|Q|$ is thus subharmonic.
Express the differential
\[
d\ln |P| = \delta_k^2 \Re \frac{dw}{w} + \Re\frac{Q'dw}{Q} - D \frac{xdx + ydy}{1 + |w|^2}.
\]
Assume that $p$ corresponds to a positive real $a$ in our $\mathbb CP^1$ affine coordinate and write
\[
\delta_k^2 \frac{dx}{a} + \Re\frac{Q'dw}{Q} - D \frac{adx}{1 + a^2} = 0, \text{ and thus, } d\ln |Q| = \Re\frac{Q'dw}{Q} = \left(\frac{Da}{1+a^2} - \frac{\delta_k^2}{a}\right)dx.
\]
If we have the inequality
\[
\frac{Da}{1+a^2} - \frac{\delta_k^2}{a} < 0,
\]
which is equivalent to
\[
a^2 < \frac{\delta_k^2}{D-\delta_k^2},
\]
then $\ln |Q|$ grows when $a$ is replaced by $a-\varepsilon$ for a sufficiently small $\varepsilon > 0$. This contradicts the choice of $a$, since $|Q|$ is subharmonic and the maximum of $|Q|$ is attained on the boundary of the disc $\{|w|\le a\}$. Hence there must be a point on the boundary $\{|w|= a\}$ such that the value of $|Q|$ at this point  is strictly greater than $|Q(a)|$. Subsequently, the value of $|P|$ at this point is strictly greater than $|P(a)|$ contradicting the initial choice of $p$.

Now we lift the points back to the sphere $S^3\subset\mathbb C^2$. Once we have $a^2\ge \frac{\delta_k^2}{D-\delta_k^2}$, the distance from $p_0=(0,1)$ to
\[
p_a=\left(\frac{a}{\sqrt{1+a^2}}, \frac{1}{\sqrt{1+a^2}}\right),
\] 
is then bounded from below as follows:
\[
\arccos |p_0\cdot p_a| = \arccos \frac{1}{\sqrt{1+a^2}} \ge \arccos \sqrt{\frac{D-\delta_k^2}{D}} = \arcsin \frac{\delta_k}{\sqrt D} \ge \arcsin \delta_k.
\]
\end{proof}

\section{Proofs of the results about the ball}

We start this section with the proof of the easier partial result.

\begin{proof}[Proof of Theorem \ref{theorem:ball-weak2}]
Due to Theorem \ref{theorem:zhao-inequality}, we know that $z=(p,q)$ is at angular distance at least $\frac{\pi}{4n}$ from the intersection of the zero set of $P(x)P(y)$ with $S^{2d-1}$. The Euclidean distance corresponding to $\frac{\pi}{4n}$ is $2\sin \frac{\pi}{8n}\ge \frac{1}{2n}$, due to concavity of $\sin x$ on $[0,\pi/2]$.

Without loss of generality, assume 
\[
|p|\le |q|, \text{ and thus, }|p| \le 1/\sqrt{2}.
\] 
Let us show that $p$ is a suitable point in the unit ball $B^d\subset\mathbb R^d$ for $C=\frac 1 8$. Assume the contrary, that is, there is a point $p_0\in \mathbb R^d$ at distance less than $\frac{1}{8n}$ from $p$ such that $P(p_0)=0$. Lift $p_0$ to a point in $S^{2d-1}$ given by
\[
(p_0, t q),\quad |p_0|^2 + t^2 |q|^2 = 1,
\]
which is equivalent to
\[
1 - t^2 = \frac{|p_0|^2 - 1 + |q|^2}{|q|^2} = \frac{|p_0|^2 - |p|^2}{|q|^2}.
\]
Hence
\[
|1 - t^2| = \frac{\left| |p_0| - |p|\right|\cdot \left| |p_0| + |p| \right|}{|q|^2} \le \frac{1}{2n}, \text{ and thus, } |1-t| \le 1 - \sqrt{1 - \frac{1}{2n}} \le \frac{\sqrt{2} - 1}{2 \sqrt 2n},
\]
with the last inequality following from the convexity of $1 - \sqrt{1 -x}$ on the interval $[0,1/2]$. The point $(p_0, tq)$ is a zero point of $P(x)P(y)$ and its Euclidean distance from $(p,q)$ is estimated as
\[
\sqrt{(p-p_0)^2 + (tq - q)^2} \le \sqrt{\frac{1}{64n^2} + (1-t)^2} \le \sqrt{\frac{1}{64n^2} + \frac{1}{32 n^2}} < \frac{1}{4n},
\]
which contradicts the lower bound of $\frac 1 {2n}$ from the initial assumption.
\end{proof}

Now we move to the proof of Theorem \ref{theorem:ball-strong}. For a reader's convenience, we first outline the structure of the proof:

\begin{itemize}
	\item
	Approximate the ball by a spherical cap of a sphere of large radius.
	
	\item
	Lift the polynomial to the sphere and multiply it by another polynomial. Choose the multiplier so that it is (almost) impossible to have maxima of the absolute value of the product on the sphere outside the original lifted cap.
	
	\item
	Apply Theorems \ref{theorem:zhao-inequality} and \ref{theorem:zhao-equality} and find a maximum point of the absolute value of the product that is sufficiently far from the zero set of the product and, therefore, from the zero set of the original polynomial.
	
	\item
	In order to make our estimate tight, carefully take the limit as the radius of the sphere tends to infinity.
\end{itemize} 

We need some preparations to describe the multiplier $G_n$ in the statement of the theorem. Consider Chebyshev polynomials of the first kind $T_n\in \mathbb R[t]$ satisfying $T_n(t)=\cos (n\arccos t)$ for $|t|\leq 1$. We first show the classical fact asserting that the sequence of the properly stretched Chebyshev polynomials of odd or even degree converges uniformly on compact sets to the sine or cosine functions, respectively.

\begin{claim}
\label{claim:uniform_converge}
The sequence of polynomials $(-1)^{m} T_{2m}(\frac{t}{2m})$ converges uniformly to $\cos t$ on compact sets. Analogously, the sequence of polynomials $(-1)^m T_{2m+1}(\frac{t}{2m+1})$ converges uniformly to $\sin t$ on compact sets.
\end{claim}

\begin{proof}
    For any $t$ with $ |t|\leq 2m$, we have
    \[
    (-1)^{m}T_{2m} \left(\frac{t}{2m}\right) 
     = (-1)^{ m } \cos \left( 2m \arccos \frac{t}{2m} \right) =\cos \left(2m\arcsin \frac{t}{2m} \right).
    \]
    %For any $t$ with $ |t|\leq k$, we have
    %\[
    %(-1)^{\lfloor k/2 \rfloor}T_{k}\left(\frac{t}{k}\right) 
     %= (-1)^{\lfloor k/2 \rfloor} \cos \left( k \arccos \left( \frac{t}{k} \right) \right) =\cos \left(\frac{(k-\lfloor k \rfloor) \pi}{2}- k\arcsin \left( \frac{t}{k} \right) \right).
    %\]
    Since the sequence $2m \arcsin (\frac{t}{2m})$ converges uniformly to $t$ on compact sets and the function $\cos t $ is $1$-Lipschitz, we obtain the desired uniform convergence to $\cos t$. 
    
    Analogously, one can prove the second uniform convergence.% of the sequence $(-1)^kT_{2k+1}(\frac{t}{2k+1})$ to $\sin t$.
\end{proof}

Recall that the cosine and sine functions can be represented as infinite products as follows
\[
    \cos t = \prod_{j=1}^{\infty} 
    \left( 1 - \left( \frac {2t} {(2j-1)\pi} \right)^2 \right) 
    \text{ and } 
    \sin t = t\prod_{j=1}^{\infty} 
    \left( 1 - \left( \frac {t} {j\pi} \right)^2 \right).
\]
The Chebyshev polynomial has a similar product representation
\[
    T_{2m}(t)=
    (-1)^m\prod_{j=1}^m \left(1- \left(\frac{t}{
        t_{j,2m}}\right)^2 \right) 
\text{ and }
    T_{2m+1}(t)=(-1)^m (2m+1)t 
    \prod_{j=1}^k \left( 1 - \left( \frac{t}{
    t_{j,2m+1}} \right)^2 \right),
\]
where $t_{1,k}<\dots<t_{\lfloor k/2\rfloor,k}$ are positive zeros of $T_k$, that is,
\begin{equation}
    \label{equation:zeros of Chebyshev}
    t_{j,k}= \sin \varphi_{j,k}, \text{ where }
\varphi_{j,k} = \begin{dcases}
        -\dfrac{\pi}{4m}+\dfrac{j\pi}{2m} & \text{ if } k=2m;\\
        \dfrac{j\pi}{2m+1} & \text{ if } k=2m+1,\\
    \end{dcases}
\end{equation}

Given positive integers $n$ and $k$ of the same parity, $n<k$, we consider the function $g_n:\mathbb R \to \mathbb R$ and the polynomial $g_{n,k}\in \mathbb R[t]$ defined as follows
\begin{equation}
\label{equation:g_n}
    g_n(t) =
    \begin{dcases}
        \prod\limits_{j=\frac{n}{2}+1}^{\infty} \left( 1 - \left( \dfrac {2t} {(2j-1)\pi} \right)^2 \right) & \text{ if } n \text{ is even};\\
        \prod\limits_{j=\frac{n+1}{2}}^{\infty} 
    \left( 1 - \left( \dfrac {t} {j\pi} \right)^2 \right) & \text{ if } n \text{ is odd};
    \end{dcases}
\end{equation}

\begin{equation}
\label{equation:g_nk}
    g_{n,k}(t) =
    \prod_{j=\lfloor n/2 \rfloor +1 }^{\lfloor k/2 \rfloor} \left( 1 - \left( \frac{t}{
    k t_{j,k}} \right)^2 \right).
\end{equation}

By~(\ref{equation:zeros of Chebyshev}), we have $2m t_{j,2m}$ and $(2m+1) t_{j,2m+1}$ converge to $j$ and $(2j-1)/2$, respectively, as $m\to \infty$. Therefore, since the product representations of $g_n$ and $g_{n,k}$ differ from corresponding trigonometric functions and polynomials $(-1)^kT_k(t/k)$, respectively, by the similar multiplicative factor, Claim~\ref{claim:uniform_converge} implies the following.

\begin{claim}
\label{claim:uniform_converge2}
The sequence of polynomials $g_{n,k}(t)$ converges to $g_n(t)$ uniformly on compact sets as $k\to\infty$.
\end{claim}

The function $G_n$ from the statement of Theorem \ref{theorem:ball-strong} is the scaled version of $g_n$, that is,
\begin{equation}
\label{equation:bigg}
G_n(t) = g_n\left(\frac{ n \pi t}{2} \right).
\end{equation}
Now we are ready to prove Theorem~\ref{theorem:ball-strong}.

\begin{proof}[Proof of Theorem \ref{theorem:ball-strong}]
For a polynomial $P\in \mathbb R[x_1,\dots, x_d]$ of degree $n$, define the continuous function $Q:\mathbb R^d\to \mathbb R$ and the polynomial $Q_k\in \mathbb R[x_1,\dots, x_d]$ of degree $k>n$, where $k$ has the same parity as $n$, as follows
\[
Q(x)=P(x)g_n\left(\frac{ n \pi |x|}{2} \right) \text{ and } Q_k(x)=P(x)g_{n,k}\left(\frac{ n \pi |x|}{2} \right).
\]
By Claim~\ref{claim:uniform_converge2}, the sequence of the polynomials $Q_k(x)$ converges to $Q(x)$ uniformly on the ball $B^d$ as $k\to\infty$.

Let the polynomial $Q_k(x)$ attain its maximal absolute value in $B^d$ at a point $p_k$. From the uniform convergence, a subsequence of $(p_k)$ converges to a point where $Q(x)$ attains its maximal absolute value in $B^d$. Hence it is sufficient to show that the distance between $p_k$ and the zero set of $P$ is bounded from below by a number converging to $1/n$ as $k\to\infty$.

Put $B^d$ in $\mathbb R^{d+1}$ with additional last coordinate $z$ in the horizontal hyperplane 
\[
    H:=\{(x,z)\in \mathbb R^{d+1}\ |\ z=0\}
\] 
and consider the $d$-dimensional sphere $S_k$ of radius $r_k=\frac{2k}{n\pi}$ centered at the origin of $\mathbb R^{d+1}$. All spheres of this sequence are of dimension $d$, so we do not indicate their dimension further on. Denote by $\spcap{r}$ the spherical cap on $S_k$ of spherical radius $r$ (that is, radius in the intrinsic metric of $S_k$) centered at the north pole $(0,r_k)$.

For the sake of brevity, denote by $Z(Q_k)$ the intersection of the zero set of the polynomial $Q_k(x)$ (independent of $z$) of degree $k$ with the sphere $S_k$. Slightly abusing notation, denote by $Z(g_{n,k})$ the intersection of the zero set of the polynomial $g_{n,k}\left(\frac{n\pi |x|}{2}\right)$ of degree $k-n$ with the sphere $S_k$. Clearly, $Z(g_{n,k})\subset Z(Q_k)$. Let $M(Q_k)$ be the set of points where $Q_k(x)$ attains its maximal absolute value on $S_k$.

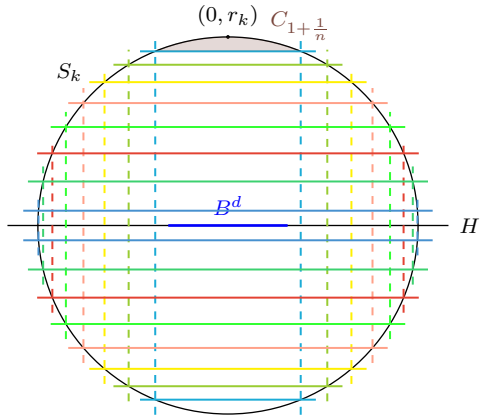
\begin{figure}[h]
\centering
\begin{tikzpicture}[scale=2.5]

%Colors for lines
\definecolor{2}{rgb}{0.47, 0.27, 0.23}
\definecolor{3}{rgb}{0.11, 0.67, 0.84}
\definecolor{4}{rgb}{0.6, 0.8, 0.2}
\definecolor{5}{rgb}{1.0, 0.94, 0.0}
\definecolor{6}{rgb}{1.0, 0.63, 0.54}
\definecolor{7}{rgb}{0.2,1,0.2}
\definecolor{8}{rgb}{0.89, 0.26, 0.2}
\definecolor{9}{rgb}{0.24, 0.82, 0.44}
\definecolor{10}{rgb}{0.28, 0.57, 0.81}

%width of horizontal colored lines except H
\def \whor {0.75pt}
%width of vertical colored dashed lines
\def \wver {0.75pt}
% the degree of Q is 2*\kk
\def \kk {10}
% the degree of P is 2*(\nn-1)
\def \nn {3}
% extra length of lines
\def \epsil {0.08}
% the degree of Q is \k
\def \k {2*\kk}
% the degree of P is \n
\def \n {2*(\nn-1)}
% pi
\def \pe {180}
% formula for angles that define the hyperplanes
\newcommand{\formula}[1]{(-\pe/(2*\k)+#1*\pe/(\k))}
% size of formulas on the picture
\def \siz {\tiny}

%S_k
\draw [line width=0.5pt] (0,0) circle (1);
\node at ({cos(135)},{sin(135)}) [above left] {\siz$S_k$};

%north pole
\draw [fill=black] (0,1) circle (0.2pt);
\node at (0,1) [above] {\siz$(0,r_k)$};

%hyperplane H
\draw [line width=0.5pt,domain={-1-2*\epsil}:{1+2*\epsil}] plot(\x,0);
\node at ({1+2*\epsil},0) [right] {\siz$H$};

%disk B
\draw [line width=1.0pt, domain={-\n*pi/(2*\k)}:{\n*pi/(2*\k)}, color=blue] plot(\x,0);
\node at (0,0) [above] {\siz\textcolor{blue}{$B^d$}};

%cylinders and horizontal hyperplanes
\foreach \s in {\nn,...,\kk}
{
  \draw [line width=\whor,domain={-sin(\formula{\s})-\epsil}:{sin(\formula{\s})+\epsil}, color={\s}] plot(\x, {cos(\formula{\s})});
  \draw [line width=\whor,domain={-sin(\formula{\s})-\epsil}:{sin(\formula{\s})+\epsil}, color={\s}] plot(\x,{-cos(\formula{\s})});
  \draw [line width=\wver,domain={-cos(\formula{\s})-\epsil}:{cos(\formula{\s})+\epsil}, color={\s},dashed] plot({sin(\formula{\s})},\x);
  \draw [line width=\wver,domain={-cos(\formula{\s})-\epsil}:{cos(\formula{\s})+\epsil}, color={\s},dashed] plot({-sin(\formula{\s})},\x);
}

%cap C_{1+1/n}
\draw [fill=2, opacity = 0.2, domain={\pe/2-\formula{\nn}}:{\pe/2+\formula{\nn}}] plot ({cos(\x)},{sin(\x)});
\node at ({sin(\formula{\nn})}, {cos(\formula{\nn})}) [above] {\siz\textcolor{2}{$C_{1+\frac{1}{n}}$}};

\end{tikzpicture}
\caption{The zero set of $g_{n,k}$ is illustrated in dashed lines; the intersection of $S_k$ with horizontal hyperplanes (colored solid lines) is $Z(g_{n,k})$.}
\label{figure:polynomialgnk}
\end{figure}

Now we can show that the set $\Zg$ partitions the sphere $S_k$ into two spherical caps $\pm \spcap{1+\frac{1}{n}}$, with centers at the north and south poles and spherical radius $1+\frac 1 n$, and $k-n-1$ spherical segments of width $\frac{2}{n}$; see Figure~\ref{figure:polynomialgnk}. Indeed, due to (\ref{equation:g_nk}) and (\ref{equation:zeros of Chebyshev}), the set $Z(g_{n,k})$ consists of horizontal $(d-1)$-dimensional spheres defined by the intersection of the sphere $S_k$ with the union of the parallel $ k-n $ horizontal hyperplanes %defined by
\[
%    \label{equation:hyperlanes}
    \left\{(x,z)\in \mathbb R^{d+1}\ \middle|\ 
    z= \pm r_k\cos \varphi_{j,k}\right\},
\]
where $ j\in\{\lfloor n/2\rfloor +1,\dots, \lfloor k/2 \rfloor \} $. Using the formula for $\varphi_{j,k}$ in~(\ref{equation:zeros of Chebyshev}), we obtain that any two consecutive $(d-1)$-dimensional spheres on $S_k$ (whose radius is $r_k$) are at the spherical distance $ \frac{\pi}{k}r_k=\frac{2}{n}$ apart. Since these $k-n$ hyperplanes form $k-n-1$ spherical segments on $S_k$, the radius of two remaining caps is $\frac 1 2 (\pi r_k - (k-n-1)\frac 2 n) = 1+\frac 1 n$. %between any two consecutive pair of them. From this description, it is clear how $\Zg$ partitions $S_k$ and what is the radius of the remaining caps $\pm \spcap{1+\frac{1}{n}}$.

Let us show that there is a point of the set $\MQ$ lying in the cap $\spcap{1+\frac{1}{n}}$. Indeed, by Theorems~\ref{theorem:zhao-inequality} and \ref{theorem:zhao-equality}, we have two possible options.
\begin{itemize}
    \item[1.] There exists a point $q_k\in \MQ$ at spherical distance strictly larger than $\frac{\pi r_k}{2k}=\frac{1}{n}$ from $\ZQ$.
    \item[2.] There is a great circle $\Sigma$ on $S_k$ containing exactly $2k$ points of $\ZQ$ and exactly $2k$ points of $\MQ$ so that the zeros and the maxima of the absolute value interlace and are equally spaced with spherical distance $\frac{\pi r_k}{2k}=\frac{1}{n}$.
\end{itemize}

In the first case, from the strict inequality on the distance the point $q_k$ cannot lie in any of the spherical segments of width $\frac{2}{n}$ bounded by spheres from $\Zg$. Therefore the point $q_k$ must lie in one of the caps $\pm \spcap{1+\frac{1}{n}}$. Since the polynomial $Q_k(x)$ is independent of $z$, we may apply the transformation $z\mapsto -z$ and assume $q_k\in \spcap{1+\frac{1}{n}}$.

In the second case, since $\Sigma\not\subset \ZQ$, the circle $\Sigma$ must intersect every sphere of $\Zg\subset \ZQ$ at two points. Otherwise, the restriction of $Q_k$ (the product of $P(x)$ of degree $k$ and $g_{n,k}(n \pi |x|/2 )$ of degree $n-k$) to the circle $\Sigma$ would have less that $2k$ roots, which is not the case. Therefore, the intersection $\Sigma\cap \spcap{1+\frac{1}{n}}$ is not empty and must contain at least one point of the set $M(Q_k)$. Denote one of them by $q_k\in M(Q_k) \cap \spcap{1+\frac{1}{n}}$. As in the first case, by Theorem~\ref{theorem:zhao-inequality}, the point $q_k$ lies at spherical distance at least $1/n$ from~$Z(Q_k)$.

In both cases, the point $q_k\in \spcap{1+\frac{1}{n}}$ must be at spherical distance at least $\frac{1}{n}$ from the boundary of $\spcap{1+\frac{1}{n}}$, because the boundary itself is a part of the set $\Zg\subset\ZQ$; see~Figure~\ref{figure:polynomialgnk}. Therefore, $q_k$ lies in the cap $\spcap{1}\subset \spcap{1+\frac{1}{n}}$ of spherical radius $1$ centered at the north pole of $S_k$.

To summarize, we have found a point $q_k\in \spcap{1}\cap \MQ$ at spherical distance at least $\frac{1}{n}$ from the intersection of $S_k$ and the zero set of the polynomial $P(x)$ (not depending on $z$) as this intersection is a subset of $\ZQ$. The projection onto the hyperplane $H$ maps $q_k$ to a point $p_k$ in the ball $B^d$. The image of the spherical cap $\spcap{1}$ under this projection is contained in the ball $B^d$ and tends to the ball $B^d$ as $k\to \infty$. The ratio between the spherical distance in $\spcap{1}$ and the Euclidean distance in its projection to $B^d$ tends to $1$ as $k\to\infty$. Hence the Euclidean distance between $p_k\in B^d$ and the zero set of $P(x)$ tends to $\delta\ge 1/n$. Therefore, we obtain the desired sequence of points $(p_k)$ maximizing the absolute value of $Q_k$.

\end{proof}

\bibliography{../Bib/karasev}
\bibliographystyle{abbrv}
\end{document}